\documentclass[12pt]{article}
\usepackage[dvips]{graphics}
\usepackage{a4,makeidx}
\usepackage{amsfonts}
\usepackage{amsmath}
\usepackage{amssymb}
\usepackage{graphics}
\usepackage{graphicx}
\usepackage{ifthen}
\usepackage{inputenc}
\usepackage{latexsym}
\usepackage{makeidx}
\usepackage{syntonly}
\usepackage{amsthm}

\def\constr#1^#2{\mathrel{\mathop{\kern 0pt#1}\limits^{#2}}}
\def\build#1_#2{\mathrel{\mathop{\kern 0pt#1}\limits_{#2}}}
  
at 8pt

\newtheorem{theorem}{Theorem}[section]
\newtheorem{definition}{Definition}[section]

\newtheorem{corollary}{Corollary}[section]
\newtheorem{proposition}{Proposition}[section]
\newtheorem{lemma}{Lemma}[section]
\newtheorem{remark}{Remark}[section]
\newtheorem{example}{Example}[section]
\numberwithin{equation}{section}

\def\fnt#1#2{\footnotetext{\kern-.3em
    {$^{\mbox{\scriptsize #1}}$}{#2}}}

\chardef\bslchar=`\\ 
\makeatletter
\newcommand{\addbslash}{\expandafter\@addbslash\string}
\def\@addbslash#1{\bslchar\@nobslash#1}
\newcommand{\nobslash}{\expandafter\@nobslash\string}
\def\@nobslash#1{\ifnum`#1=\bslchar\else#1\fi}
\newcommand{\ntt}{\normalfont\ttfamily}
\def\@boxorbreak{\leavevmode
  \ifmmode\hbox\else\ifdim\lastskip=\z@\penalty9999 \fi\fi}
\DeclareRobustCommand{\cs}[1]{\@boxorbreak{\ntt\addbslash#1\@empty}}
\makeatother 

\title{\bf ON A JOINT $(m;(q_1,...,q_d))$-PARTIAL ISOMETRIES AND A JOINT $m$-INVERTIBLE $d$- TUPLE OF OPERATORS ON A HILBERT SPACE }
\author{ Ould Ahmed  Mahmoud Sid Ahmed
\\
  Mathematics Department, College of Science. Aljouf
University\\Aljouf 2014. Saudi Arabia
 \\
 sidahmed@ju.edu.sa}

\begin{document}

\maketitle
\begin{abstract}
 For $d \in \mathbb{N}$ with $d \geq 1$, let ${\bf
\large T}=(T_1,T_2,....,T_d) \in \mathcal{B}(\mathcal{H})^d$ with
$T_j:\mathcal{H}\longrightarrow \mathcal{H}$ be a tuple of commuting
bounded linear operators. Let  $\alpha = (\alpha_1;
\alpha_2,...,\alpha_d) \in \mathbb{Z}_+^d,q=(q_1,q_2,...,q_d) \in
\mathbb{Z}_+^d$ denote tuples of nonnegative integers respectively,
and set $|\alpha|: = \displaystyle\sum_{1\leq j\leq d}|\alpha_j|,$
$\alpha!: =\alpha_1!...\alpha_d!$.  Further, define ${\bf \large
T}^\alpha:= T_1^{\alpha_1} T_2^{\alpha_2}...T_d^{\alpha_d}$. A
bounded linear $d$-tuple of commuting operators  ${\bf \large
T}=(T_1,T_2,...,T_d)$ acting on a Hilbert  space $\mathcal{H}$ is
called an $(m; (q_1,q_2,...,q_d))$- partial isometry, if
\begin{eqnarray*} {\bf \large T}^q\Bigg(\sum_{0\leq k\leq m}(-1)^k\binom{m}{k}\sum_{|\alpha|=k}\frac{k!}{\alpha!}{\bf \large T}^{*\alpha}{\bf \large  T}^{\alpha}\Bigg)=0.
\end{eqnarray*}
The aim of the present paper is,firstly we study the concepts of
$(m;(q_1,...,q_d))$-partial isometries on a Hilbert space; secondly,
we introduce the notion of $m$-invertibility  of tuples of operators
as a natural generalization of the $m$-invertibility in single
variable operators.
\end{abstract}
{\bf Keywords.} $m$-isometric tuple, partial isometry, Left
$m$-inverse, Right $m$-inverse, \par\vskip 0.2 cm \quad \qquad \quad
joint spectrum, joint approximate spectrum.\par \vskip 0.2 cm
\noindent {\bf Mathematics Subject Classification (2010).} Primary:
17A13. Secondary:  47A16.
\begin{center}
\section{INTRODUCTION AND TERMINOLOGIES}
\end{center}
Let $\mathcal{H}$ be an infinite dimensional  separable complex
Hilbert space  and denote by $\mathcal{B}(\mathcal{H})$ the algebra
of all bounded linear operators from $\mathcal{H}$ to $\mathcal{H}$.
For  $T \in \mathcal{B}(\mathcal{H}) $ we shall write
$\mathcal{N}(T)$ , $\mathcal{R}(T)$  and  $\mathcal{N}(T)^\bot$ for
the null space, the range of $T$  and the othogonal complement of
$\mathcal{N}(T)$ respectively. $I = I_{\mathcal{H}}$ being the
identity operator. In what follows $\mathbb{N},\mathbb{Z}_+$ and
$\mathbb{C}$ stands the sets of positive integers, nonnegative
integers and complex numbers respectively. Denote by
$\overline{\lambda}$ the complex conjugate of a complex number
$\lambda$ in $\mathbb{C}$. We shall henceforth shorten $\lambda
I_{\mathcal{H}}-T$ by $\lambda-T$. The spectrum, the point spectrum,
the approximate point spectrum of an operator $T$ are denoted by
$\sigma(T), \sigma_p(T) , \sigma_{ap}(T) $ respectively. $T^*$ means
the adjoint of $T.$ \par \vskip 0.2 cm \noindent The study of tuples
of commuting operators was the subject of a wide literature carrying
out many resemblances with the single case. Some developments toward
this subject have been done in
 $[4],[5],[10],[11],[12],[13],[14],[26]$ , $[27]$  and the references therein.\par\vskip 0.2 cm \noindent
Our aim in this paper is to extend the notions of $m$-partial
isometries ($[24]$) and $(m,q)$-partial isometries ($[21]$) for
single variable operators to the tuples of commuting operators
defined on a complex Hilbert space.\par\vskip 0.2cm \noindent Some
notational explanation is necessary before we begin.
 For $d \in \mathbb{N}$
 $(d \geq 1)$, let ${\bf \large T}=(T_1,T_2,....,T_d) \in
\mathcal{B}(\mathcal{H})^d$ be a tuple of commuting bounded linear
operators. Let  $\alpha = (\alpha_1, \alpha_2,...,\alpha_d) \in
\mathbb{Z}_+^d$ denote tuples of nonnegative integers multi-indices)
and set $|\alpha|: = \displaystyle\sum_{1\leq j\leq d}|\alpha_j|,$
$\alpha!: =\alpha_1!...\alpha_d!$. Further, define $T^\alpha:=
T_1^{\alpha_1} T_2^{\alpha_2}...T_d^{\alpha_d}$ where
$T_j^{\alpha_j}$ denotes the product of $T_i$ times itself
$\alpha_j$ times .\par \vskip 0.2 cm \noindent Now let $p(z,
\overline{z})$ be a non-commutative complex polynomial in $z = (z_1,
..., z_d) $ and $\overline{z} = (\overline{z_1}, ... ,
\overline{z_d})$ given by $p(z, \overline{z})
=\displaystyle\sum_{\alpha,\beta}a_{\alpha,\beta}z^\alpha
\overline{z }^\beta$. If ${\bf T}$ denote an $d$-tuple of bounded
linear operators on a Hilbert space then one can associate with
$p(z, \overline{z}$) an operator polynomial $p({\bf  \large T}, {\bf
 \large T}^*)$
$$p({\bf \large T}, {\bf  \large T}^* )=\sum_{\alpha,\beta}a_{\alpha,\beta}{\bf \large  T}^{*\alpha}
{\bf  \large T}^\beta$$  by replacing $z$ and $\overline{z}$ by
${\bf  \large T} = (T_1, ... , T_d)$ and ${\bf  \large T}^*=
(T_1^*,...,T_d^*)$ respectively.
\par \vskip 0.2 cm \noindent
One of the most important subclasses, of the algebra of all bounded
linear operators acting on a Hilbert space, the class of partial
isometries operators. An operator $T \in \mathcal{B}(\mathcal{H})$
is said to be isometry if $T^*T=I$ and partial isometry if
$TT^*T=T.$ In recent years this classes has been generalized, in
some sense, to the larger sets of operators  so-called
$m$-isometries and $m$-partial isometries. An operator $T \in
\mathcal{B}(\mathcal{H})$ is said to be $m$-isometric  for some
integer $ m \geq 1 $ if it satisfies the operator equation

\begin{eqnarray}
\sum_{0\leq k\leq m}(-1)^k\binom{m}{k}T^{*m-k}T^{m-k}=0.
\end{eqnarray}
 It is immediate that $T$ is $m$-isometric if and only if \begin{equation}\qquad \sum_{0\leq k\leq
m}(-1)^k\binom{m}{k}\|T^{m-k}x\|^2=0
\end{equation}
for all $x  \in \mathcal{H}.$ Major work on $m$-isometries has been
done in a long paper consisting of three parts by Agler and Stankus
$([1, 2, 3])$ and have since then attracted the attention of several
other authors (see for example $ [7],[8],[9],[16]).$ More recently a
generalization of these operators to $m$-partial isometries  has
been studied in the paper of A.Saddi and the present author in
$[24]$ and by the present author in $[21]$.\par \vskip 0.2 cm
\noindent  An operator $T \in \mathcal{B}(\mathcal{H})$ is called an
$m$-partial isometry (see $[24]$) if \begin{eqnarray} T\Bigg(
T^{*m}T^{m}- \binom{m}{1} T^{*m-1}T^{m-1}+ \binom{m}{2}
T^{*m-2}T^{m-2}-....+(-1)^mI\Bigg)=0.
\end{eqnarray}
and it is an $(m,q)$-partial isometry for $m\in \mathbb{N}$ and $q
\in \mathbb{Z}_+$ (see $([21]))$ if
\begin{eqnarray} T^q\Bigg( T^{*m}T^{m}- \binom{m}{1}
T^{*m-1}T^{m-1}+ \binom{m}{2} T^{*m-2}T^{m-2}-....+(-1)^mI\Bigg)=0.
\end{eqnarray}
Gleason and Richter in $[17]$ extend the notion of $m$-isometric
operators to the case of commuting $d$-tuples of bounded linear
operators on a Hilbert space. The defining equation for an
 $m$-isometric tuple ${\bf \large T} = (T_1, ..., T_d) \in
 \mathcal{B}(\mathcal{H})^d$
reads:
\begin{equation}
\sum_{0\leq k\leq
m}(-1)^{m-k}\binom{m}{k}\sum_{|\alpha|=k}\frac{k!}{\alpha!}{\bf
\large T}^{*\alpha}{\bf \large  T}^{\alpha}=0
\end{equation}
or equivalently
\begin{equation}
\sum_{0\leq k\leq
m}(-1)^{m-k}\binom{m}{k}\sum_{|\alpha|=k}\frac{k!}{\alpha!}\|{\bf
\large  T}^{\alpha}x\|^2=0 \;\;\hbox{for all} \;x \in \mathcal{H}.
\end{equation}
\par\vskip 0.2 cm \noindent Recently, P.H.W.Hoffmann and M.Mackey
in $[20]$  introduced the concept of $(m,p)$-isometric tuples on
normed space. A tuple of commuting linear operators ${\bf\large T}
:= (T_1, ..., T_d)$ with $T_j : X \longrightarrow X $ (normed space)
is called an $(m, p)$-isometry (or an $(m, p)$-isometric tuple) if,
and only if, for given $m \in \mathbb{N}$ and $p \in (0, \infty)$,
\begin{equation}
\sum_{0\leq k\leq
m}(-1)^{m-k}\binom{m}{k}\sum_{|\alpha|=k}\frac{k!}{\alpha!}\|{\bf
\large  T}^{\alpha}x\|^p=0\;\;\hbox{for all }\;\; x \in X.
\end{equation}
 \begin{definition} Let
${\bf\large T}=(T_1,T_2,...,T_d)\in \mathcal{B}(\mathcal{H})^d$ be a
tuple of operators.\par\vskip 0.2 cm \noindent (1) If
$T_iT_j=T_jT_i\;\;\;1\leq i,\;j\leq d$, we say that ${\bf\large T}$
is a commuting tuple. \par\vskip 0.2 cm \noindent (2) If
$T_iT_j=T_jT_i$,\;$T_iT_j^*=T_j^*T_i\;\;\;1\leq i\not=j\leq d,$  we
say that ${\bf\large T}$ is a doubly commuting tuple.
\end{definition}

 \begin{definition}$([18])$ A commuting tuple
${\bf T}=(T_1,T_2,...,T_d) \in \mathcal{B}(\mathcal{H})^d$ is
called:
\par \vskip 0.2 cm \noindent (1) \; matricially  quasinormal if
$T_i$ commutes with $T_j^*T_k$ for all $i,j,k \in
\{\;1,2,...d\;\}.$\par \vskip 0.2 cm \noindent (2) \; jointly
quasinormal if $T_i$ commutes with $T_j^*T_j$ for all $i,j\in
\{\;1,2,...,d\;\}$ and
\par \vskip 0.2 cm \noindent (3)\; spherically
quasinormal if $T_j$ commutes with $|{\bf \large
T}|:=\bigg(\displaystyle\sum_{1\leq j\leq d}T_j^*T_j\bigg)$ for all
$j=1,2,...,d.$
\end{definition}
 \par\vskip 0.2 cm \noindent If
$\mathcal{M}$ is a common invariant subspace of $\mathcal{H}$ for
each $T_j\in \mathcal{B}(\mathcal{H}),$ then
$\displaystyle{\bf\large
T}_{|\mathcal{M}}=({T_1}_{|\mathcal{M}},{T_2}_{|\mathcal{M}},...,{T_d}_{|\mathcal{M}})$
denote an $d$-tuple of compressions of $\mathcal{M}.$\par\vskip 0.2
cm \noindent The contents of this paper are the following.
Introduction and terminologies are described in the first part. The
second part is devoted to the study of some basic properties of the
class of $(m;(q_1,...,q_d))$-partial isometries tuples. Several
spectral properties of some $(m;(q_1,...,q_d))$-partial isometries
are obtained in section three; concerning the joint point
spectrum,the joint approximate spectrum and the spectral radius. In
the fourth section we present some results concerning the  left
$m$-inverses and the right $m$- inverses for tuples of operators.

\begin{center}
\section{JOINT $(m;(q_1,...,q_d))$-PARTIAL ISOMETRIES $d$- TUPLE OF OPERATORS
}\end{center}

 In this Section, we  introduce and study some basic properties of an  joint $(m;(q_1,...,q_d))$-partial
isometry  operators tuples. All of these results are fairly
straightforward generalizations of the corresponding single variable
results that were proved in $[21]$ and $[24]$.\par\vskip 0.2 cm
\noindent
 The notion of  a joint $(m;(q_1,...,q_d))$-partial
isometry  is a natural higher dimensional generalization of the
notion of $(m,q)$- partial isometry.
\par \vskip 0.2 cm
\begin{definition} Given  $m \in \mathbb{N} $ and  $q=(q_1,q_2,...,q_d) \in \mathbb{Z}_+^d,$. An commuting operator  $d$-tuple
${\bf  \large T} \in \mathcal{B}(\mathcal{H})^d$ is called an joint
$(m;(q_1,...,q_d))$-partial isometry (or joint
$(m;(q_1,...,q_d))$-partial isometric
 $d$-tuple ) if and only if\begin{eqnarray*} {\bf
 \large T}^q\Bigg(\sum_{0\leq k\leq
m}(-1)^k\binom{m}{k}\sum_{|\alpha|=k}\frac{k!}{\alpha!}{\bf
 \large T}^{*\alpha}{\bf  \large T}^{\alpha}\Bigg)=0.
\end{eqnarray*}
\end{definition}
\begin{remark}\begin{enumerate}\item
Every $m$-isometric $d$-tuple of operators on $\mathcal{H}$ is a
joint $(m;(q_1,q_2,...,q_d))$-partial isometry $d$-tuple.
\item Every $(m;(q_1,q_2,...,q_d))$-partial isometry $d$-tuple of
operators ${\bf\large T}=(T_1,T_2,...,T_d)$ such that ${\bf\large
T}$ is entry-wise invertible,  ${\bf\large T}$ is an $m$-isometric
$d$-tuple.
\end{enumerate}
\end{remark}

\begin{remark}
 If $d=2$, let ${\bf \large
T}=(T_1,T_2) \in \mathcal{B}(\mathcal{H})^2 $ be a commuting
operator $2$-tuple, we have that
\par \vskip 0.2 cm \noindent (i)\;${\bf  \large T}$ is a joint $(1;(1,1))$-partial isometry  pair if
$$T_1T_2\bigg(I-T_1^*T_1-T_2^*T_2\bigg)=0.$$
\par \vskip 0.2 cm \noindent (ii)\;${\bf \large
T}$ is a joint $(2;(1,1))$-partial isometry  pair if
$$T_1T_2\bigg(I-2T_1^*T_1-2T_2^*T_2+T_1^{*2}T_1^2+T_2^{*2}T_2^2+2T_1^*T_2^*T_1T_2\bigg)=0.$$
\end{remark}
\begin{remark}
 Let ${\bf \large
T}=(T_1,T_2,...,T_d) \in \mathcal{B}(\mathcal{H})^d $ be a commuting
operator $d$-tuple.Then  ${\bf \large T}$ is an joint
$(1;(1,1,..,1))$-partial isometry if and only if
$$T_1...T_d\bigg(I-T_1^*T_1-T_2^*T_2-...-T_d^*T_d\bigg)=0.$$
\end{remark}
\begin{example} Consider
$T=\displaystyle\left( \begin{array}{ccc}
0&0&1\\
0&0&0\\
\frac{1}{\sqrt{2}}&\frac{1}{\sqrt{2}}&0\\
\end{array}
\right)\in \mathcal{B}(\mathbb{C}^3)$ and let  ${\bf \large
T}=\displaystyle\bigg(\frac{1}{\sqrt{d}}T,\frac{1}{\sqrt{d}}T,...,\frac{1}{\sqrt{d}}T\bigg)\in
\mathcal{B}(\mathcal{\mathbb{C}}^3)^d$. It is easy to see that ${\bf
\large T}$ is a joint $(1;(1,1,...,1))$-partial isometry $d$-tuple.
\end{example}
\begin{remark}
If   ${\bf \large T}=(T_1,T_2,...,T_d) \in
\mathcal{B}(\mathcal{H})^d $ be an doubly commuting $d$-tuple of
operators on $\mathcal{H}$.Then ${\bf\large T}$ is an joint
$(1;1,1,...,1)$-partial isometry if and only if ${\bf \large
T}^*:=(T_1^*,T_2^*,...,T_d^*)$ is so.
\end{remark}
The following example of a joint $(m; (q_1,...,q_d))$-partial
isometry is adopted form $[20].$
\begin{example}
Let  $S\in \mathcal{B}(\mathcal{H})$ be an $(m,q_1)$-partial
isometry operator,$d \in \mathbb{N}$ and
$\lambda=(\lambda_1,\lambda_2,...,\lambda_d)\in (\mathbb{C}^d,\;
\|.\|_2)$ with  $$\|\lambda\|_2^2=\sum_{1\leq j\leq
d}|\lambda_j|^2=1.$$ Then the operator tuple ${\bf\large
T}=(T_1,T_2,...,T_d)$ with $T_j=\lambda_jS$ for $j=1,2,..,d$ is an
joint $(m,(q_1,q_2,...,q_d))$-partial isometry $d$-tuple.\par \vskip
0.2 cm \noindent  In fact,it is clair that $T_iT_j=T_jT_i$ for all
$1\leq i;\;j\leq d.$ Further, by the multinomial expansion, we get
\begin{eqnarray*}\bigg(|\lambda_1|^2+|\lambda_2|^2+...+|\lambda_d|^2\bigg)^k
&=&\sum_{\alpha_1+\alpha_2+...+\alpha_d=k}\binom{k}{\alpha_1,\alpha_2,...,\alpha_d}\prod_{1\leq
i\leq d}|\lambda_i|^{2\alpha_i}
\\&=&\sum_{|\alpha|=j}\frac{k!}{\alpha!}|\lambda^\alpha|^2.\end{eqnarray*}
Thus, we have
\begin{eqnarray*}
{\bf \large T}^q\sum_{0\leq j\leq
m}(-1)^k\binom{m}{k}\sum_{|\alpha|=k}\frac{k!}{\alpha!}{\bf\large
T}^{*\alpha}{\bf\large T}^\alpha&=&{\bf \large T}^q\sum_{0\leq k\leq
m}(-1)^k\binom{m}{k}\sum_{|\alpha|=k}\frac{k!}{\alpha!}\prod_{1\leq
j \leq d}|\lambda|^{2\alpha_j}
S^{*|\alpha|}S^{|\alpha|}\\&=&\prod_{1\leq j\leq
d}\lambda_j^{q_j}S^{|q|}\sum_{0\leq k\leq
m}(-1)^k\binom{m}{k}S^{*k}S^{k}\\&=& \prod_{1\leq j\leq
d}\lambda_j^{q_j}S^{|q|-q_1}\underbrace{S^{q_1}\sum_{0\leq k\leq
m}(-1)^k\binom{m}{k}S^{*k}S^{k}}_{=0}\\&=&0.
\end{eqnarray*}
Consequently ${\bf \large T}$ is an joint $(m;
(q_1,...,q_d))$-partial isometry $d$-tuple as required.
\end{example}
The following example shows that the question about joint
$(m;(q_1,q_2,...,q_d))$-partial isometry for $d$-tuple is non
trivial.There exists a $d$-tuple  of commuting operators ${\bf
\large T}=(T_1,T_2,...,T_d) \in \mathcal{B}(\mathcal{H})^d$ such
that each $T_j$ is $(m,q_j)$-partial isometry for $j=1,2,...,d$, but
${\bf \large T}=(T_1,T_2,...,T_d)$ is not an joint
$(m;(q_1,q_2,...,q_d))$-partial isometry.
\begin{example} Let us consider $\mathcal{H}=\mathbb{C}^3$ and
define  $T_1=\left( \begin{array}{ccc}
0&i&0\\
0&0&i\\
i&0&0\\
\end{array}
\right) $ and $T_2=\left( \begin{array}{ccc}
1&0&0\\
0&1&0\\
0&0&1\\
\end{array}
\right).$ It is straightforward that $T_1$ and $T_2$
commute.Moreover,$T_1$ and $T_2$ are $(2;1)$-partial isometry but
$(T_1,T_2)$ is  not a $(2;(1,1))$-partial isometry.
\end{example}
\begin{lemma} Let $\mathbb{S}_d$ be the group of permutation on $d$ symbols $\{1,2,...,d\}$
and let  ${\bf\large T}=(T_1,T_2,...,T_d)\in
\mathcal{B}(\mathcal{H})^d$ be an $d$-tuple  of commuting operators.
If ${\bf\large T}$  is an joint $(m;(q_1,q_2,...,q_d))$-partial
isometry , then for every $\sigma \in \mathbb{S}_d$, ${\bf\large
T_{\sigma}}:=(T_{\sigma(1)},T_{\sigma(2)},...,T_{\sigma(d)})$ is an
joint $(m;(q_{\sigma(1)},q_{\sigma(2)},...,q_{\sigma(d)}))$-partial
isometry.
 \end{lemma}
 \begin{proof} It follows from the condition that $\displaystyle\prod_{1\leq j\leq d}T_j=\prod_{1\leq j\leq d}T_{\sigma{(j)}}$ and the identity
 $$\prod_{1\leq j\leq d}T_j^{q_j}\sum_{0\leq k\leq m}(-1)^k\binom{m}{k}\sum_{|\alpha|=k}\frac{k!}{\alpha!}\prod_{1\leq j\leq d}T_j^{*\alpha_j}\prod_{1\leq j\leq d}T_j^{\alpha_j}=0.$$

 \end{proof}
\begin{theorem} Let $m\in \mathbb{N}$ and $q=(q_1,q_2,...,q_d)\in
\mathbb{Z}_+^d$. Let ${\bf \large T}=(T_1,T_2,...,T_d) \in
\mathcal{B}(\mathcal{H})^d$ be an commuting $d$-tuple operators such
that $\mathcal{N}(T^q)$ is a reducing subspace for $T_j$ for all
$j=1,2,...,d$. Then the following properties are equivalent.\par
\vskip 0.2 cm \noindent (1)\;${\bf T}$ is an joint
$(m;(q_1,...,q_d))$-partial isometry.\par \vskip .2 cm \noindent (2)
$$\sum_{0\leq k \leq m}(-1)^m\binom{m}{k}\sum_{|\alpha|=k}\frac{k!}{\alpha!}\|{\bf
T}^\alpha {\bf T}^{*q}x\|^2=0,\;\;\hbox{for all}\;x \in
\mathcal{H}.$$
\end{theorem}
\begin{proof}
First , assume that ${\bf  \large T}$ is an joint
$(m;(q_1,q_2,...,q_d))$-partial isometry. We have that for all $x
\in \mathcal{H}$
\begin{eqnarray*}
&&{\bf T}^q\sum_{0\leq k \leq
m}(-1)^k\binom{m}{k}\sum_{|\alpha|=k}\frac{k!}{\alpha!}{\bf
T}^{*\alpha}{\bf T}^{\alpha}{\bf
T}^{*q}x=0\\
&\Longrightarrow& \langle  {\bf T}^q\sum_{0\leq k \leq
m}(-1)^k\binom{m}{k}\sum_{|\alpha|=k}\frac{k!}{\alpha!}{\bf
T}^{*\alpha}{\bf T}^{\alpha}{\bf
T}^{*q}x,\;x\rangle=0\\&\Longrightarrow& \sum_{0\leq k \leq
m}(-1)^k\binom{m}{k}\sum_{|\alpha|=k}\frac{k!}{\alpha!}\|{\bf
T}^{\alpha}{\bf T}^{*q}x\|^2=0.
\end{eqnarray*} Thus, (2) holds.\par\vskip 0.2cm\noindent  To prove
the converse, assume that the equality in (2) holds. It follows
that,
\begin{eqnarray*}
 &&\langle T^q\sum_{0\leq k \leq
m}(-1)^k\binom{m}{k}\sum_{|\alpha|=k}\frac{k!}{\alpha!}T^{*\alpha}T^{\alpha}T^{*q}x,\;x\rangle=0,\;\forall
\;x \in \mathcal{H}\\&\Longrightarrow& T^q\sum_{0\leq k \leq
m}(-1)^k\binom{m}{k}\sum_{|\alpha|=k}\frac{k!}{\alpha!}T^{*\alpha}T^{\alpha}T^{*q}x=0,\forall
\;x \in \mathcal{H}.
\end{eqnarray*}
Hence,  $$T^q\sum_{0\leq k \leq
m}(-1)^k\binom{m}{k}\sum_{|\alpha|=k}\frac{k!}{\alpha!}T^{*\alpha}T^{\alpha}=0
\;\;\hbox{on}\;\;
 \overline{\mathcal{R}(T^{*q})}=\mathcal{N}(T^{q})^\bot.$$ As
 $\mathcal{N}(T^q)$ is a reducing subspace for each $T_j$ $(1\leq j \leq d)$, we have that  $$T^q\displaystyle\sum_{0\leq k \leq
m}(-1)^k\binom{m}{k}\sum_{|\alpha|=k}\frac{k!}{\alpha!}T^{*\alpha}T^{\alpha}=0\;\;\hbox{on}\;\;
\mathcal{N}(T^q)$$ and hence, $$T^q\displaystyle\sum_{0\leq k \leq
m}(-1)^k\binom{m}{k}\sum_{|\alpha|=k}\frac{k!}{\alpha!}T^{*\alpha}T^{\alpha}=0.$$
\end{proof}
\noindent The following corollary is a immediate consequence of
Theorem 2.1.
\begin{corollary}
Let $m\in \mathbb{N}$ and $q=(q_1,q_2,...,q_d)\in \mathbb{Z}_+^d$.
Let ${\bf \large T}=(T_1,T_2,...,T_d) \in
\mathcal{B}(\mathcal{H})^d$ be an commuting $d$-tuple operators such
that $\mathcal{N}({\bf \large T}^q)$ is a reducing subspace for each
$T_j$, $1\leq j \leq d$. Then the following properties are
equivalent
\begin{enumerate}
\item ${\bf \large T}$ is an joint $(m;q_1,...,q_d))$-partial isometry.

\item$ {\bf \large T}|_{\mathcal{N}({\bf \large T}^q)^\perp}:=\bigg({\bf \large T}_1|_{\mathcal{N}({\bf \large T}^q)^\perp},T_2|_{\mathcal{N}({{\bf \large T}^q)^\perp}},...,T_d|_{\mathcal{N}({\bf \large T}^q)^\perp}\bigg)$ is an $m$-isometric tuple.
\end{enumerate}
\end{corollary}
\begin{remark}
It easy to see that every $(m;( 1,1,...,1))$-partial isometry
$d$-tuple of commuting operators is an $(m; (q_1,...,q_d))$-partial
isometry $d$-tuple.
\end{remark}
In the following theorem  we show that by imposing certain
conditions on $(m;(q_1,...,q_d))$-partial isometry operator it
becomes $m$-partial isometry.

\begin{theorem}
If ${\bf T}=(T_1,T_2,...T_d)\in \mathcal{B}(\mathcal{H})^d$ is an
joint $(m;(q_1,...,q_d))$-partial isometry such that
$\mathcal{N}(T_j)=\mathcal{N}(T_j^2)$ for each $j$, $1\leq j\leq d$,
then ${\bf T}$ is an joint $(m;(1,...,1))$-partial isometry.
\end{theorem}
\begin{proof} By the assumption we have  for $j=1,...,d$ that
$\mathcal{N}(T_j)=\mathcal{N}(T_j^n)$ for all positive integer $n$.
It follows that \begin{eqnarray*} {\bf T}^q\Bigg(\sum_{0\leq k \leq
m}(-1)^k\binom{m}{k}\sum_{|\alpha|=k}\frac{k!}{\alpha!}{\bf
T}^{*\alpha}{\bf T}^{\alpha } \Bigg)=0
\end{eqnarray*}
implies \begin{eqnarray*} \prod_{1\leq j\leq d}T_j\Bigg( \sum_{0\leq
k \leq m}(-1)^k\binom{m}{k}\sum_{|\alpha|=k}\frac{k!}{\alpha!}{\bf
T}^{*\alpha}{\bf T}^{\alpha}\Bigg)=0.
\end{eqnarray*}
\end{proof}
The following proposition generalized Proposition 3.1 in $[24]$.
\begin{proposition}
If ${\bf T}=(T_1,T_2,...T_d)\in \mathcal{B}(\mathcal{H})^d$ is a
jointly quasinormal and an joint $(m;(1,...,1))$-partial isometry,
then ${\bf \large T}$ is a joint $(1;(1,...,1))$-partial isometry.
\end{proposition}
\begin{proof}
Since ${\bf T}=(T_1,T_2,...T_d)$ is a matricially quasinormal and an
joint $(m;(1,...,1))$-partial isometry, it follows that
$$ \prod_{1\leq j\leq d}T_j\Big(I-\sum_{1\leq j\leq d}T_j^*T_j\Big)^m=0.$$ A straightforward  computation using this
last equation yields that
$$ \prod_{1\leq j\leq d}T_j\Big(I-\sum_{1\leq j\leq d}T_j^*T_j\Big)=0.$$ The proof is complete.
\end{proof}
\begin{definition}

Let ${\bf\large T}=(T_1,T_2,...,T_d)$ and ${\bf\large
S}=(S_1,S_2,...,S_d)$ are two commuting $d$-tuple on of operators on
a common Hilbert space $\mathcal{H}$. We said that ${\bf\large S}$
is unitary equivalent to ${\bf\large T}$
 if there exists an unitary operator $V\in \mathcal{B}(\mathcal{H})$ such that
 $${\bf\large S}=(S_1,S_2,...,S_d)=(V^*T_1V,V^*T_2V,...,V^*T_dV).$$
\end{definition}
\begin{proposition}
Let ${\bf\large T}=(T_1,T_2,...,T_d)$ and ${\bf\large
S}=(S_1,S_2,...,S_d) \in \mathcal{B}(\mathcal{H})^d$
 are two commuting $d$-tuple of operators such that ${\bf \large S}$ is unitary equivalent to  ${\bf\large T}$ , then
 ${\bf\large T}$ is a joint $(m,(q_1,q_2,...,q_d)$-partial isometry if and only if   ${\bf\large S}$ is a joint $(m,(q_1,q_2,...,q_d)$-partial isometry.
\end{proposition}
\begin{proof} Suppose that ${\bf\large S}$ and ${\bf \large T}$ are unitary equivalent,that is there exists a unitary operator
 $V \in \mathcal{B}(\mathcal{H})$  such that $S_j=V^*T_jV\;\;(1\leq j\leq d)$. Since $T_iT_j=T_jT_i$; it follows that
$$(V^*T_jV)(V^*T_iV)=(V^*T_iV)(V^*T_jV)\; \;\hbox{for all}\;\; 1\leq i,j\leq d.$$\\Using the observations above, we get the following identity
\begin{eqnarray*}
{\bf\large S}^q\sum_{0\leq k\leq
m}(-1)^k\binom{m}{k}\sum_{|\alpha|=k}\frac{k!}{\alpha!}{\bf\large
S}^{*\alpha}{\bf\large S}^\alpha&=& V^*{\bf\large T}^qV\sum_{0\leq
k\leq
m}(-1)^k\binom{m}{k}\sum_{|\alpha|=k}\frac{k!}{\alpha!}V^*{\bf\large
T}^{*\alpha}{\bf\large T}^\alpha V\\&=& V^*\bigg({\bf\large
T}^q\sum_{0\leq k\leq
m}(-1)^k\binom{m}{k}\sum_{|\alpha|=k}\frac{k!}{\alpha!}V^*{\bf\large
T}^{*\alpha}{\bf\large T}^\alpha\bigg) V
\end{eqnarray*}
\end{proof}
In the proof of the following theorem , we need the following
formula \begin{remark} For $n,d,k_1,k_2,...,k_d \in \mathbb{N}$ with
$k_1+...+k_d=n$, $n\geq 1$ and $d\geq 2$, we have
$$\binom{n}{k_1...k_d}=\sum_{1\leq j\leq d}\binom{n-1}{k_1...k_j-1...k_d}.$$
\end{remark}
\begin{theorem}
Let ${\bf\large T}=(T_1,T_2,...,T_d\in \mathcal{B}({\cal H})^d$ be
an $(m;(q_1,q_2,...,q_d))$-partial isometry  $d$-tuple of operators
such that ${\cal N}({\bf\large T}^q)$
 is a reducing subspace for each $T_j$ for $1\leq j\leq d$. Then ${\bf \large T}$ is an
 $(m+n;(q_1,q_2,...,q_d))$-partial isometry $d$-tuple for $n \in \mathbb{N}$.
\end{theorem}
\begin{proof} To prove that ${\bf \large T}$ is an
 $(m+n;(q_1,q_2,...,q_d))$-partial isometry, it suffices to prove
 that ${\bf \large T}$ is an
 $(m+1;(q_1,q_2,...,q_d))$-partial isometry.\par\vskip .2 cm
 \noindent Indeed, we have
\begin{eqnarray*}
&&\sum_{0\leq k \leq
m+1}(-1)^k\binom{m+1}{k}\sum_{|\alpha|=k}\frac{k!}{\alpha!}\|{\bf
\large T}^\alpha {\bf \large T}^{*q}x\|^2\\&=&\|{\bf \large
T}^{*q}x\|^2\!+\!\sum_{1\leq k \leq m}(-1)^k\big[
\binom{m}{k}\!+\!\binom{m}{k-1}\big]\sum_{|\alpha|=k}\frac{k!}{\alpha!}\|{\bf
{\bf \large T}}^\alpha {\bf \large T}^{*q}x\|^2
\!-\!(-1)^m\sum_{|\alpha|=m+1}\frac{(m+1)!}{\alpha!}\|{\bf \large
T}^\alpha {\bf \large T}^{*q}x\|^2\\&=&\sum_{0\leq k \leq
m}(-1)^k\binom{m}{k}\sum_{|\alpha|-k}\frac{k!}{\alpha!}\|{\bf \large
T}^\alpha{\bf \large T}^{*q}x\|^2-\sum_{0\leq k\leq
m-1}(-1)^k\binom{m}{k}\sum_{|\alpha|=k+1}\frac{(k+1)!}{\alpha!}\|{\bf
\large T}^\alpha {\bf\large
T}^{*q}x\|^2\\&&-(-1)^m\sum_{|\alpha|=m+1}\frac{(m+1)!}{\alpha!}\|{\bf
\large T}^\alpha{\bf\large T}^{*q} x\|^2\\&=&-\sum_{0\leq k\leq
m-1}(-1)^k\binom{m}{k}\sum_{|\alpha|=k+1}\frac{k!(\alpha_1+...+\alpha_d)}{\alpha_1!.\alpha_2....\alpha_d!}\|{\bf
\large T}^\alpha {\bf\large
T}^{*q}x\|^2\\&&-(-1)^m\sum_{|\alpha|=m+1}\frac{m!(\alpha_1+...+\alpha_d)}{\alpha_1!.\alpha_2....\alpha_d!}\|{\bf
\large T}^\alpha{\bf\large T}^{*q} x\|^2\\&=&-\sum_{1\leq j \leq
d}\sum_{0\leq k\leq
m-1}(-1)^k\binom{m}{k}\sum_{|\alpha|=k+1}(-1)^k\binom{m}{k}\frac{k!\alpha_j}{\alpha_1!.\alpha_2!....\alpha_d!}\|T^{\alpha_1}...T_j^{\alpha_j-1}T_{j+1}^{\alpha_{j+1}}...T_d^{\alpha_d}T_j{\bf\large
T}^{*q}x\|^2
\\&&-(-1)^m\sum_{1\leq j
\leq
d}\sum_{|\alpha|=m+1}\frac{m!\alpha_j}{\alpha_1!.\alpha_2!....\alpha_d!}\|T^{\alpha_1}...T_j^{\alpha_j-1}T_{j+1}^{\alpha_{j+1}}...T_d^{\alpha_d}T_j{\bf\large
T}^{*q}x\|^2\\&=& -\sum_{1\leq j \leq d}\sum_{0\leq k\leq
m-1}(-1)^k\binom{m}{k}\sum_{|\beta|=k}(-1)^k\binom{m}{k}\frac{k!}{\beta!}\|{\bf
T}^{\beta}T_j{\bf\large T}^{*q}x\|^2
\\&&-(-1)^m\sum_{1\leq j
\leq d}\sum_{|\alpha|=m}\frac{m!}{\beta!}\|{\bf
T}^{\beta}T_j{\bf\large T}^{*q}x\|^2\\&=&-\sum_{1\leq j\leq
d}\sum_{0\leq k \leq
m}(-1)^k\binom{m}{k}\sum_{|\beta|=k}\frac{k1}{\beta!}\|{\bf\large
T}^\beta T_j{\bf\large T}^{*q}x\|^2\\&=&0.
\end{eqnarray*}
This completes the proof.
\end{proof}
\begin{proposition}  Let ${\bf \large T}=(T_1,...,T_d) \in \mathcal{B}(\mathcal{H})^d$ be an commuting  $d$-tuple of operators such that  $\mathcal{N}(T^q)$ is a reducing subspace for $T_j$ for all
$j=1,2,...,d$.
 If ${\bf \large T}$ is an joint  $(m+1;(q_1,q_2,...q_d))$-partial isometry
and  an joint  $(m;(q_1,q_2,...q_d))$-partial isometry on
$\displaystyle\bigcap_{1\leq j\leq d}\mathcal{R}(T_j)$,
  then ${\bf T}$ is a joint
$(m;(q_1,q_2,...q_d))$-partial isometry on $\mathcal{H}$.
\end{proposition}
\begin{proof}
A simple computation shows that
\begin{eqnarray*}
&&\sum_{0\leq k\leq
m+1}(-1)^k\binom{m+1}{k}\sum_{|\alpha|=k}\frac{k!}{\alpha!}\|{\bf\large
T}^\alpha {\bf\large T}^{*q}x\|^2\\&=&\sum_{0\leq k\leq
m}(-1)^k\binom{m}{k}\sum_{|\alpha|=k}\frac{k!}{\alpha!}\|{\bf\large
T}^\alpha {\bf\large T}^{*q}x\|^2-\sum_{1\leq j\leq d}\sum_{0\leq
k\leq
m}(-1)^k\binom{m}{k}\sum_{|\alpha|=k}\frac{k!}{\alpha!}\|{\bf\large
T}^\alpha T_j{\bf\large T}^{*q}x\|^2.
\end{eqnarray*}
Thus complete the proof by invoking Corollary 2.1.
\end{proof}
\begin{proposition} Let ${\bf \large T}=(T_1,T_2,...,T_d) \in \mathcal{B}(\mathcal{H})^d$ be
an joint $(m;(q_1,q_2,...,q_d))$-partial isometry $d$-tuple. Then
${\bf\large T}$ is an joint $(m+1;(q_1,...,q_d))$-partial isometry
$d$-tuple if and if ${\bf\large T}$ satisfy the following identity
\begin{equation}\sum_{1\leq j\leq d}\sum_{0\leq k\leq
m}(-1)^k\binom{m}{k}\sum_{|\alpha|=k}\frac{k!}{\alpha!}\|{\bf\large
T}^\alpha T_j{\bf\large T}^{*q}x\|^2=0\;\;\hbox{for all}\;\;x \in
\mathcal{H}.\end{equation}
\end{proposition}
\begin{proof}
Assume that $T$ is an joint $(m;(q_1,q_2,...,q_d)$-partial isometry
$d$-tuple and an $(m+1;(q_1,q_2,...,q_d)$-partial isometry
$d$-tuple. In this case, we get
\begin{eqnarray*}
0&=&{\bf \large T}^q\sum_{0\leq k\leq m+1}
(-1)^k\binom{m+1}{k}\sum_{|\alpha|=k}\frac{k!}{\alpha!}{\bf \large
T}^{*\alpha}{\bf\large T}^\alpha\\&=&{\bf \large T}^q\sum_{0\leq
k\leq m} (-1)^k\binom{m}{k}\sum_{|\alpha|=k}\frac{k!}{\alpha!}{\bf
\large T}^{*\alpha}{\bf\large T}^\alpha-{\bf \large T}^q\sum_{1\leq
j\leq d}\sum_{0\leq k\leq m}
(-1)^k\binom{m}{k}\sum_{|\beta|=k}\frac{k!}{\beta!}T^*_{j}{\bf
\large T}^{*\beta}{\bf\large T}^\beta T_j.
\end{eqnarray*}
Then, we obtain that  $${\bf \large T}^q\sum_{1\leq j\leq
d}\sum_{0\leq k\leq m}
(-1)^k\binom{m}{k}\sum_{|\beta|=k}\frac{k!}{\beta!}T^*_{j}{\bf
\large T}^{*\beta}{\bf\large T}^\beta T_j=0,$$ and hence
$$\sum_{1\leq
j\leq d}\sum_{0\leq k\leq m}
(-1)^k\binom{m}{k}\sum_{|\beta|=k}\frac{k!}{\beta!}\|{\bf\large
T}^\beta T_j{\bf \large T}^qx\|^2=0\;\;\;\hbox{for all}\;\;x\in
\mathcal{H}.$$ Conversely assume that ${\bf \large T}$ is an
$(m;q_1,q_2,...,q_d)$-partial isometry $d$-tuple satisfy
$(2.1)$.\par \vskip 0.2 cm \noindent From equation $(2.1)$ it
follows that
\begin{eqnarray*} 0&=&
{\bf \large T}^q \sum_{1\leq j\leq d}\sum_{0\leq k\leq
m}(-1)^k\binom{m}{k}\sum_{|\beta|=k}\frac{k!}{\beta!}T_j^{*}{\bf\large
T}^{*\beta}{\bf\large T}^{\beta}T_j\\&=&{\bf \large T}^q \sum_{1\leq
j\leq d}\sum_{0\leq k\leq
m}(-1)^k\binom{m}{k}\sum_{|\alpha|=k+1}\frac{k!.\alpha_j}{\alpha!}{\bf\large
T}^{*\alpha}{\bf\large T}^{\alpha}\\&=&{\bf \large T}^q \sum_{0\leq
k\leq
m}(-1)^k\binom{m}{k}\sum_{|\alpha|=k+1}\frac{(k+1)!}{\alpha!}{\bf\large
T}^{*\alpha}{\bf\large T}^{\alpha}.\end{eqnarray*} On the other
hand, we have that \begin{eqnarray*}&&{ \large\bf T}^q \sum_{0\leq
k\leq
m+1}(-1)^k\binom{m+1}{k}\sum_{|\alpha|=k}\frac{k!}{\alpha!}{\bf\large
T}^{*\alpha}{\bf\large T}^{\alpha}\\&=& {\bf\large T}^q \sum_{0\leq
k\leq
m}(-1)^k\binom{m}{k}\sum_{|\alpha|=k}\frac{k!}{\alpha!}{\bf\large
T}^{*\alpha}{\bf\large T}^{\alpha}-{\bf \large T}^q \sum_{0\leq
k\leq
m}(-1)^k\binom{m}{k}\sum_{|\alpha|=k+1}\frac{(k+1)!}{\alpha!}{\bf\large
T}^{*\alpha}{\bf\large T}^{\alpha}\\&=&0.\end{eqnarray*} The proof
is complete.
\end{proof}

\begin{center}
\section{SPECTRAL PROPERTIES OF A JOINT $(m;(q_1,...,q_d))$-PARTIAL ISOMETRIES
d-TUPLES}\end{center}

Spectral properties of commuting $d$-tuples received important
attention during last decades. Systematic investigations have been
carried out to extend known results for single operators to
commuting n-tuples. For more details, the interested reader is
referred to $[6],[10],[11], [12],[13],[14] ,[25], [27]$ and the
references therein.\par \vskip 0.2 cm \noindent First, we
recapitulate very briefly the following definitions.\par \vskip 0.2
cm
\begin{definition}
Let ${\bf \large T}=(T_1,T_2,...,T_d)$ be an $d$-tuple of operators
on a complex Hilbert space $\mathcal{H}$.\begin{enumerate}\item A
point $\lambda=(\lambda_1,\lambda_2,...,\lambda_d)\in \mathbb{C}^d$
is called a point eigenvalue of ${\bf T}$ if there exists a non zero
vector $x\in \mathcal{H}$ such that
$$(T_j-\lambda_j)x=0\;\;\hbox{for}\;\;j=1,2,...,d.$$ \noindent Or equivalently if
there exists a non-zero vector $x \in \mathcal{H}$such that $x \in
\displaystyle\bigcap_{1\leq j\leq
d}\mathcal{N}(T_j-\lambda_j)$,i.e.; $$\sigma_p({\bf\large
T})=\{\lambda \in \mathbb{C}^d:\;\bigcap_{1\leq j\leq
d}\mathcal{N}(T_j-\lambda_j)\not=\{0\}\}.$$
\item The joint point spectrum, denoted by $\sigma_p({\bf \large
T})$  of ${\bf \large T}$ is the set of all joint eigenvalues of
${\bf \large T}.$
\end{enumerate}
\end{definition}
\begin{definition}
For a commuting $d$-tuple ${\bf\large T}=(T_1,...,T_d) \in
\mathcal{B}(\mathcal{H})^d$. A number
$\lambda=(\lambda_1,\lambda_2,...,\lambda_d)\in \mathbb{C}^d$ is in
the joint approximate point spectrum $\sigma_{ap}({\bf \large T})$
if and only if there exists a sequence $(x_n)_n$ such that
$$(T_j -\lambda_j ) x_n \longrightarrow 0  \;\;\hbox{as}
\;\;n\longrightarrow \infty \;\;\hbox{for every }\;\;j=1,...,d.
$$\end{definition}
\begin{lemma}($[17]$)
Let ${\bf\large T}=(T_1,T_2,...,T_d)\in \mathcal{B}(\mathcal{H})^d$
be a commuting tuples of bounded operators. Then
$$\sigma_{ap}({\bf\large T})=\bigg\{\lambda=(\lambda_1,\lambda_2,...,\lambda_d)\in \mathbb{C}^d: \;\exists\;(x_n)_n\subset \mathcal{H}\;\;\hbox{such that}\;\;
\\\lim_{n\longrightarrow \infty}\sum_{1\leq j\leq
d}\|(T_j-\lambda_j)x_n\|=0\bigg\}.$$
\end{lemma}
\begin{definition}($[27]$)
The Taylor spectrum  of commuting d-tuple ${\bf
\large}=(T_1,...,T_d)\in \mathcal{B}(\mathcal{H})^d$ is the set  of
all complex $d$-tuple $\lambda=(\lambda_1,...,\lambda_d)\in
\mathbb{C}^d$ with the property that the translated $d$-tuple
$(T_1-\lambda_1,...,T_d-\lambda_d)$ is note invertible.The symbol
$\sigma({\bf  \large T})$ will stand for the Taylor spectrum of
${\bf \large T}$.
\end{definition}
\begin{remark}( $[27]$ )
Let ${\bf \large T}=(T_1,T_2,...,T_d)\in \mathcal{B}(\mathcal{H})^d$
de an $d$-tuple of commuting operators on $\mathcal{H}.$
$(\lambda_1,\lambda_2,...,\lambda_d)\notin \sigma({\bf \large T})$
if there exist operators $U_1,...,U_d,V_1,...,V_d \in
\mathcal{B}(\mathcal{H}))$ such that
$$\sum_{1\leq k \leq d}U_k(T_k-\lambda_k I)=I\;\hbox{and}\;\;
\sum_{1\leq k \leq d}(T_k-\lambda_k I)V_k =I.$$ The spectral radius
of ${\bf\large T}$ is $$r({\bf\large T})=\max\{\|\lambda\|_2,\;\;
{\lambda \in \sigma({\bf\large T})}\}$$ where
$\|\lambda\|_2=\bigg(\displaystyle\sum_{1\leq j\leq
d}|\lambda_j|^2\bigg)^2.$
\end{remark}
\begin{proposition}($[23]$,\;Lemma 3.1.1)
Let ${\bf\large T}=(T_1,T_2,...,T_d)\in \mathcal{B}(\mathcal{H})^d$
be a commuting tuples of bounded operators. Then the following ate
equivalent
\par\vskip .2 cm\noindent (1)\;There exists $\delta>0$,such that $\|T_1x\|+...+\|T_dx\|\geq \delta\|x\|
$ for all $x\in \mathcal{H},$ \par\vskip .2 cm\noindent (2)\; There
exists ${\bf\large S}=(S_1,...,S_d)\in \mathcal{H}(\mathcal{H})^d$
such that $S_1T_1+ S_2T_2+...+S_dT_d=I_{\mathcal{H}}.$ \par\vskip .2
cm\noindent (3)\; There is no sequence $(x_n)_n\subset
\mathcal{H}$:$\|x_n\|=1$ such that
$\displaystyle\lim_{n\longrightarrow \infty}\|T_jx_n\|=0$ for
$j=1,2,...,d.$

\end{proposition}
\par \vskip 0.2 cm
\noindent In the following results we examine some spectral
properties of a joint $(m;(q_1,...,q_d))$-partial isometries.That
extend the case of single variable $m$-partial isometries studied in
$[24].$\par\vskip 0.2 cm  \noindent We put
$$\mathbb{B}(\mathbb{C}^d):=\{ \lambda=(\lambda_1,...,\lambda_d)\in
\mathbb{C}^d \;/ \; \|\lambda\|_2=\bigg(\sum_{1\leq j \leq
d}|\lambda_j|^2\bigg)^{\frac{1}{2}}<1\;\}$$  and
$$\partial\mathbb{B}(\mathbb{C}^d):=\{
\lambda=(\lambda_1,...,\lambda_d)\in \mathbb{C}^d \;/ \;
\|\lambda\|_2 =\bigg(\sum_{1\leq j \leq
d}|\lambda_j|^2\bigg)^{\frac{1}{2}} =1\;\}$$\par\vskip 0.2 cm
\noindent

In ( $[17]$, Lemma 3.2), the authors proved that If ${\bf \large T}$
is a $m$-isometric tuple, then the joint approximate point spectrum
of $T$ is in the boundary of the unit ball
$\mathbb{B}(\mathbb{C}^d)$.
 This is not true
for an joint $(m;(q_1,q_2,...,q_d))$-partial isometry tuple. For
example, on $\mathcal{B}(\mathbb{C}^2)^d$  the operator ${\bf \large
T}=(T,0,...,0) $ where $T$ is the matrix operator $T = \left(
                                                     \begin{array}{cc}
                                                       a & 0 \\
                                                       1 & 0 \\
                                                     \end{array}
                                                   \right)$
with $ \displaystyle |a|^2 = \frac{1+\sqrt{5}}{2}$ is an
$(2;(1,0,0,...,0))$-partial isometry.It is clear that with
$\sigma({\bf \large S})= \{0, a\}\times\{0\}\times...\times \{0\}.$
\par \vskip 0.2 cm \noindent
However, if in addition assume that $T_j$ reduces
$\mathcal{N}({\bf\large T^q})$ for $1\leq j\leq d$, we obtain the
following result.

\begin{theorem}
Let ${\bf\large  T}=(T_1,T_2,...,T_d) \in
\mathcal{B}(\mathcal{H})^d$ be an joint $(m;(q_1,...,q_d))$-partial
isometry of $d$-tuple of operators such that $\mathcal{N}({\bf
\large T^q})$ is a reducing subspace for each  $T_j$  $( 1\leq j\leq
d)$. Then $\sigma_{ap}({\bf \large T}) \subset
\partial \mathbb{B}(\mathbb{C}^d )\displaystyle\cup \big[0\big]$ where $$\big[0
\big]:=
 \{ (\lambda_1,\lambda_2,...,\lambda_d) \in
\mathbb{C}^d:\;\prod_{1\leq k\leq d}\lambda_k=0\}.$$
\end{theorem}
\begin{proof}  Let $\lambda =(\lambda_1,\lambda_2,...,\lambda_d)\in \sigma_{ap}({\bf \large T})$, then there exists a
sequence $(x_n)_{n\geq 1} \subset \mathcal{H}$, with $||x_n||=1$
such that $({ T_j}-\lambda_j I)x_n \longrightarrow 0$ for all
$j=1,2,...,d$. Since for $\alpha_j>1$,
$$T_j^{\alpha_j}-\lambda_j^{\alpha_j}=(T_j-\lambda_j)\sum_{1\leq k\leq \alpha_j}\lambda_j^{k-1}T_j^{\alpha_j-k}$$ By induction, for  $\alpha
\in \mathbb{Z}_+^d$, we have $$ (T^\alpha-\lambda^\alpha I)=
\sum_{1\leq k\leq d}\bigg( \prod_{i\leq
k}\lambda_i^{\alpha_i}\bigg)\bigg(
T_j^{\alpha_j}-\lambda_j^{\alpha_i}\bigg)\prod_{i>k}T_i^{\alpha_i}.$$
Since, $\mathcal{R}({\bf \large T}^q) \subset \mathcal{N}({\bf
\large T}^q)^\perp$ we have  from Corollary 2.1 that , for all
$n\geq 1$
\begin{eqnarray*}
0 &=&\lambda^q \langle \sum_{0\leq k\leq
m}(-1)^{k}\binom{m}{k}\sum_{|\alpha|=k}\frac{k!}{\alpha!}{\bf\large
T}^{*\alpha}{\bf\large T}^\alpha {\bf \large T}^qx_n\;, x_n \rangle \\
&=&\lambda^q \langle \sum_{0\leq k\leq
m}(-1)^{k}\binom{m}{k}\sum_{|\alpha|=k}\frac{k!}{\alpha!}{\bf
T}^{*\alpha}{\bf\large T}^\alpha({\bf\large  T}^q-\lambda^q)x_n,\; x_n \rangle + \lambda^{2q}\sum_{0\leq k\leq m}(-1)^{k}\binom{m}{k}\sum_{|\alpha|=k}\frac{k!}{\alpha!}||{\bf T}^{\alpha}x_n||^2\\
&=& \lambda^q \langle \sum_{0\leq k \leq m}(-1)^{k}\binom{m}{k}{\bf T}^{*\alpha}{\bf T}^\alpha({\bf T}^q-\lambda^qx_n| x_n \rangle \\ &+& \lambda^{2q}\bigg\{\sum_{0\leq k \leq m}(-1)^{k}\binom{m}{k}\sum_{|\alpha|=k}\frac{k!}{\alpha!}\bigg(||({\bf T}^{\alpha}-\lambda^\alpha)x_n||^2+2Re \langle({\bf T}^{\alpha}-\lambda^\alpha)x_n|\lambda^\alpha x_n\rangle +|\lambda^\alpha|^{2}\bigg)\bigg\}\\
\end{eqnarray*}
as $ \big({\bf T}^\alpha-\lambda^\alpha I\big)x_n \to 0$ as $n
\longrightarrow \infty$ for all $\alpha \in \mathbb{Z}_+^d$  we
obtain that
$$0=\lambda^q\sum_{0\leq k \leq m}(-1)^k\binom{m}{k}\sum_{|\alpha|=k}\frac{k!}{\alpha!}|\lambda^\alpha|^2=\lambda^q(1-\|\lambda\|_2^2)^m,$$  where
$| \lambda|=\bigg(\displaystyle\sum_{1\leq k\leq
d}|\lambda_k|^2\bigg)^{\frac{1}{2}}$ Then $\lambda^q= 0 $ or
$\|\lambda\|_2=1$. This implies that
$$\lambda \in \{ (\lambda_1,\lambda_2,...,\lambda_d) \in
\mathbb{C}^d:\;\prod_{1\leq k\leq d}\lambda_k=0\}\;\;\hbox{or}\;\;
\lambda \in \partial \mathbb{B}(\mathbb{C}^d).
$$
\end{proof}
\begin{corollary} If ${\bf\large T}=(T_1,T_2,...,T_d) \in \mathcal{B}(\mathcal{H})^d$ is an $(m;(q_1,q_2,...,q_d))$-partial isometry  $d$-tuple of operators such that  $\mathcal{N}({\bf\large T}^q)$ is a reducing subspace for each $T_j$
(  $1\leq j\leq d$). Then $r({\bf \large T})=1$.  In particular
$\sigma({\bf \large T}) \subset \mathbb{\partial B}(\mathbb{C}^d)$
or $\sigma_a({\bf \large T})= \overline{\mathbb{B}}(\mathbb{C}^d).$
\end{corollary}
\noindent{\bf Proof.}\; It is known (see for example $[25]$ that the
convex envelopes of all spectra coincide. Thus from Theorem 4.1 we
have that the approximate point spectrum of the tuple ${\bf \large
T}=(T_1,T_2,...,T_d)$ is contained in the boundary of the unit
ball,it follows that $r(T)=1.$ \\
In the other hand, $\rho({\bf\large T) \cap
\mathbb{B}}(\mathbb{C}^d)$ is both open and closed subset of the
domain $\mathbb{B}(\mathbb{C}^d).$ Consequently we find
$\sigma({\bf\large T}) \subset \mathbb{\partial B}(\mathbb{C}^d)$ or
$\sigma({\bf\large T})= \overline{\mathbb{B}}(\mathbb{C}^d).$

\par \vskip 0.2 cm \noindent
We have also, the following properties.
\begin{proposition} Let ${\bf \large T}=(T_1,T_2,...,T_d)\in \mathcal{B}(\mathcal{H})^d$ be an joint  $(m;(q_1,q_2,...,q_d))$-partial isometry $d$-tuple such that $\mathcal{N}({\bf T}^q)$ is a reducing subspace for   $T_j$ $(1\leq j\leq
d$). The following properties hold.
\begin{enumerate}
\item If $ \lambda=(\lambda_1,\lambda_2,...,\lambda_d) \in \sigma_{ap}({\bf T})\backslash \big[0 \big]$ then $
\overline{\lambda}=(\overline{\lambda_1},\overline{\lambda_2},...,\overline{\lambda_d})\in
\sigma_{ap}({\bf T}^*).$ \item If $
\lambda=(\lambda_1,\lambda_2,...,\lambda_d) \in \sigma_p({\bf
T})\backslash \big[0 \big]$ then $\overline{\lambda}\in
\sigma_p({\bf T}^*).$
\item Eigenvectors of ${\bf T}$ corresponding to distinct eigenvalues are
orthogonal.
\end{enumerate}
\end{proposition}

\begin{proof} \begin{enumerate}  \item Let $\lambda= (\lambda_1,\lambda_2,...,\lambda_d)\in \sigma_{ap}({\bf T})\backslash
\{ (\lambda_1,\lambda_2,...,\lambda_d) \in
\mathbb{C}^d:\displaystyle\prod_{1\leq k\leq d}\lambda_k=0\}$,
choose a sequence $(x_n)_{n}\subset\mathcal{H},$ such that
$\|x_n\|=1$ and  $(T_j-\lambda_j ) x_n\longrightarrow 0$ for all
$j=1,2,...,d$. Following similar arguments it is easy to see that
for all $\alpha_j\geq 0.$
$(T_j^{\alpha_j}-\lambda_j^{\alpha_j})x_n\longrightarrow 0$ as
$n\longrightarrow \infty$  and
$$(T^\alpha-\lambda^\alpha)x_n\longrightarrow 0.$$
On the other hand
\begin{eqnarray*}
{\bf \large T}^{*\alpha}{\bf T}^\alpha({\bf \large
T}^q-\lambda^q)x_n&=&{\bf \large T}^{*\alpha}{\bf \large T}^\alpha
{\bf \large T}^qx_n-\lambda^q{\bf \large T}^{*\alpha}{\bf \large
T}^\alpha x_n\\&=& {\bf \large T}^{*\alpha}{\bf \large T}^\alpha
{\bf \large T}^qx_n-\lambda^q{\bf \large T}^{*\alpha}\big({\bf
\large T}^\alpha-\lambda^\alpha\big)x_n+\lambda^q{\bf \large
T}^{*\alpha}\lambda^\alpha x_n \longrightarrow 0.
\end{eqnarray*}
Since ${\bf \large T}$ is an joint $(m;(q_1,q_2,...,q_d))$-partial
isometry ,we observe that
$$\lambda^q\sum_{1\leq k\leq m}(-1)^k\binom{m}{k}\sum_{|\alpha|=k}\frac{k!}{\alpha!}\big(\lambda {\bf \large T}^*)^\alpha x_n\longrightarrow 0.$$
 and hence,
$$\lambda^q\bigg( {\bf \large I}-\sum_{1\leq j \leq d}\lambda_j T_j^*\bigg)^m x_n\longrightarrow
0,$$ Using the fact that
$\lambda=(\lambda_1,\lambda_2,...,\lambda_d)\in \sigma_{ap}({\bf
\large T})\backslash [0]$, we get
$$\bigg( \underbrace{\sum_{1\leq j\leq d}|\lambda_j|^2}_{=1}I-\sum_{1\leq j \leq d}\lambda_j T_j^*\bigg)^mx_n\longrightarrow 0,$$
or equivalently
$$\bigg(\sum_{1\leq j\leq d}\lambda_j\big(\overline{\lambda_j}-T_j^*\big) \bigg)^mx_n\longrightarrow 0
.$$

We deduce that
$$\sum_{1\leq j \leq d}\lambda_jI_{\mathcal{H}}.(\overline{\lambda_j}-T_j^*) $$ is not bounded below and  in view of Proposition 3.1,
it follows that there exists a sequence $(x_n)_n\subset \mathcal{H}$
such that $\|x_n\|=1$ and
$$\lim_{n\longrightarrow \infty}\|(T_j^*-\overline{\lambda}_j)x_n\|=0\;\;\hbox{for}\;\;j=1,2,...,d.$$
So we get
 $(\overline{\lambda}_1,
\overline{\lambda}_2...,\overline{\lambda}_d)\in \sigma_{ap}({\bf
\large T^*})$ and the proof of this implication is over.\item

 Let $\lambda =(\lambda_1,\lambda_2,...,\lambda_d)\in \sigma_p({\bf\large T})\backslash [0]$,
there exists a non zero vector $x\in \mathcal{H}$ such that
$$T_jx=\lambda_j x \;\;\hbox{for}\;\; j=1,2,..,d.$$ By using a similar argument as in
$1$ we show
 $(T_j^*-\overline{\lambda}_jI) x=0$ for  j  $(1\leq j \leq d)$ from which it follows
that
$\overline{\lambda}=(\overline{\lambda}_1...,\overline{\lambda}_d)\in
\sigma_p({\bf\large T}^*).$ \item  Let $\lambda
=(\lambda_1,\lambda_2,...,\lambda_d)$ and
$\mu=(\mu_1,\mu_2,...,\mu_d)$ be distinct eigenvalues of ${\bf
\large T}$. Assume that $$T_jx=\lambda_j  x\;\;\hbox{ and}\;\;
T_jy=\mu_j y\;\;\hbox{for}\;\;j=1,2,...,d.$$ Then
\begin{eqnarray*}
0&=&\lambda^q\big \langle  \displaystyle\sum_{0\leq k\leq
m}(-1)^k\binom{m}{k}\sum_{|\alpha|=k}\frac{k!}{\alpha!} {\bf \large
T}^{*\alpha}{\bf \large T}^\alpha x ,\; y \big\rangle
\\
& =& \lambda^q \displaystyle\sum_{0\leq k\leq
m}(-1)^k\binom{m}{k}\sum_{|\alpha|=k}\frac{k!}{\alpha!}\big(
\lambda.\overline{\mu}\big)^\alpha\langle
x,\;y\rangle\\&=&\bigg(1-\sum_{1\leq j\leq
d}\lambda_j\overline{\mu}_j\bigg)^m\langle x,\;y\rangle.
\end{eqnarray*}
where
$\lambda.\overline{\mu}=(\lambda_1\overline{\mu}_1,\lambda_2\overline{\mu}_2,,...,\lambda_d\overline{\mu}_d).$

 Since  $1-\displaystyle\sum_{1\leq j\leq
d}\lambda_j\overline{\mu}_j \not=0$, we obtain that $\langle
x\;|\;y\rangle =0.$
\end{enumerate}
\end{proof}

\begin{lemma}
 Let ${\bf\large T}=(T_1,T_2,...,T_d)\in \mathcal{B}(\mathcal{H})^d$ be an joint $(m;(q_1,q_2,...,q_d))$-partial isometry such that $\mathcal{N}({\bf\large T^q})$ is a reducing subspace for $T_j,$   $j=1,..,d$.
 Let  $ \lambda =(\lambda_1,...,\lambda_d)$ and $
\mu=(\mu_1,...\mu_d) \in \sigma_{ap}({\bf\large T})$  such that
$\lambda-\mu \notin [0]$. If
 $(x_n)_n$ and  $(y_n)_n$ are two sequences of unit vectors in $\mathcal{H}$ such that
$$\|(T_j-\lambda_j) x_n\|\longrightarrow 0 \;\hbox{and} \; \|(T_j-\mu_j)y_n\| \longrightarrow 0\;\;(\hbox{as}\;\;
n \longrightarrow \infty) \;\;\hbox{for all}\;\;j=1,2,...,d ,$$ then
we have
\begin{equation}\label{eqnconllem}
\langle x_n |\; y_n\rangle \longrightarrow 0  \;\; (\hbox{as}\;\;
n\longrightarrow \infty).
\end{equation}
\end{lemma}
\begin{proof}  Assume that $\mu \notin [0].$ Then from part 1. of  Proposition 3.2
we have  that $\|(T_j^*-\overline{\mu_j})y_n\| \longrightarrow 0$ as
$n\longrightarrow \infty,\;j=1,2,...,d.$ Hence,for all $j=1,2,...,d$
$$(\lambda_j-\mu_j)\langle x_n|\; y_n\rangle =-\langle (T_j-\lambda_j)x_n|\; y_n\rangle
+ \langle x_n |\;(T_j-\mu_j)^*y_n\rangle \longrightarrow 0,\;\; n
\longrightarrow \infty,$$ which implies \eqref{eqnconllem} in view
of $\lambda - \mu \notin [0]$ and the proof is complete.\end{proof}
\section{JOINT LEFT $m$-INVERSE AND JOINT RIGHT $m$-INVERSE OF TUPLE OF OPERATORS}
An operator  $T\in \mathcal{B}(\mathcal{H})$  is said to be left
invertible if there is an operator $S\in \mathcal{B}(\mathcal{H})$
such that $ST = I_{\mathcal{H}}$, where $I_{\mathcal{H}}$,denotes
the identity operator. The operator $S$ is called a left inverse of
$T$.
 An operator  $T\in \mathcal{B}(\mathcal{H})$   is said to be right invertible if there is an operator  $R\in \mathcal{B}(\mathcal{H})$
such that $TR = I_{\mathcal{H}}$. The operator $R$ is called a right
inverse of $T$.\par\vskip 0.2 cm \noindent The left and right
$m$-invertibility of operator have been introduced by the present
author in  $[22]$ and by B.P.Duggal and V. M\"{u}ller in
$[15]$.\par\vskip 0.2 cm \noindent  Given a positive integer $m$. A
bounded linear operator $T$ is called left $m$-invertible (resp.
right $m$-invertible) if there exists a bounded linear operator $S$
such that
$$\sum_{0\leq k \leq m}(-1)^{m-k}\binom{m}{k}S^kT^k=0
\bigg(resp. \sum_{0\leq k \leq
m}(-1)^{m-k}\binom{m}{k}T^kS^k=0\bigg).$$ The $m$-invertibility have
been extensively studied in the recent paper $[19]$ by C.Gu.
\par \vskip 0.2 cm \noindent The following definition generalize the
definition of left $m$-invertibility and right $m$-invertibility of
a single operator to tuple of operators.
\begin{definition}
Let  ${\bf \large T}=(T_1,T_2,...,T_d) \in
\mathcal{B}(\mathcal{H})^d $ be an  commuting $d$-tuple of operators
on $\mathcal{H}$,we say that ${\bf\large T}$ is a joint left
$m$-invertible (resp. joint right $m$-invertible) for some integer
$m\geq 1$, if there exists a commuting $d$-tuple operators ${\bf
\large S}=(S_1,S_2,...,S_d) \in \mathcal{B}(\mathcal{H})^d $ such
that
$$\sum_{0\leq k\leq m}(-1)^{m-k}\binom{m}{k}\sum_{|\alpha|=k}\frac{k!}{\alpha!}{\bf\large S}^\alpha {\bf\large
T}^\alpha=0\;\;$$ $$ \bigg(\;\hbox{resp}.\;\;\sum_{0\leq k\leq
m}(-1)^{m-k}\binom{m}{k}\sum_{|\alpha|=k}\frac{k!}{\alpha!}{\bf\large
T}^\alpha {\bf\large S}^\alpha=0\bigg).$$ ${\bf\large S}$ is called
a left (resp.right) $m$-inverse of ${\bf\large T}.$\par\vskip 0.2 cm
\noindent We say that ${\bf \large T}=(T_1,T_2,...,T_d) \in
\mathcal{B}(\mathcal{H})^d $ is $m$-invertible $d$-tuple of
commuting operators if it has both a left $m$-inverse and a right
$m$-inverse.

\end{definition}
An interesting example of a left $m$-invertible commuting tuple
operator is that of an $m$-isometric tuple operator.
\begin{remark} It is  clear that ${\bf\large S}$ is a
left $m$-inverse of ${\bf\large T}$ if and only if ${\bf\large T}^*$
is a left $m$-inverse of ${\bf\large S}^*.$\end{remark}
\begin{example}
Let $T_1=\left(
           \begin{array}{cc}
             1 & 1 \\
             0 & 1 \\
           \end{array}
         \right)$ and $S_1=\left(
           \begin{array}{cc}
             1 & -1 \\
             0 & 1 \\
           \end{array}
         \right)
$. Then the pair ${\bf\large T}=(T_1,T_1)$ is $m$-invertible tuple
with $m$-inverse ${\bf\large S}=(S_1,S_2)$ in
$\mathcal{B}(\mathbb{C}^2)^2$.
\end{example}
\begin{remark}

\begin{enumerate}
\item ${\bf \large S}=(S_1,S_2,...,S_d) $ is a joint left inverse (or 1-inverse) of ${\bf \large T}=(T_1,T_2,...,T_d)$ if and only if
$$S_1T_1+S_2T_2+...+S_dT_d=I_{\mathcal{H}}.$$
\item  ${\bf \large S}=(S_1,S_2,...,S_d) $ is a joint right inverse (or 1-inverse) of ${\bf \large T}=(T_1,T_2,...,T_d)$ if and only if
$$T_1S_1+T_2S_2+...+S_dT_d=I_{\mathcal{H}}.$$
\end{enumerate}
\end{remark}
$${\beta}_{m}({\bf \large S},{\bf \large T})=\displaystyle\sum_{0\leq k\leq m}(-1)^{m-k}\binom{m}{k}\sum_{|\alpha|=k}\frac{k!}{\alpha!}{\bf\large S}^\alpha {\bf\large T}^\alpha.$$
\begin{lemma}
Il ${\bf \large T}=(T_1,T_2,...,T_d) $ and ${\bf \large
S}=(S_1,S_2,...,S_d) \in \mathcal{B}(\mathcal{H})^d $ are commuting
$d$-tuples of operators,then we have the following equality
$$\beta_{m+1}({\bf \large S},{\bf \large T})=-\beta_{m}({\bf \large S},{\bf \large T})+\sum_{1\leq j \leq d}S_j\beta_{m}({\bf \large S},{\bf \large T})T_j$$

\end{lemma}
\begin{proof}
\begin{eqnarray*}
\beta_{m+1}({\bf\large S},{\bf\large T})&=&
(-1)^{m+1}I_{\mathcal{H}}\!+\!\sum_{1\leq k \leq m}(-1)^k\big[
\binom{m}{k}\!+\!\binom{m}{k-1}\big]\sum_{|\alpha|=k}\frac{k!}{\alpha!}{\bf\large
S}^\alpha{\bf\large T}^\alpha
\!+\!\sum_{|\alpha|=m+1}\frac{(m+1)!}{\alpha!}{\bf\large
S}^\alpha{\bf\large T}^\alpha\\&=& -\sum_{0\leq k \leq
m}(-1)^k\binom{m}{k}\sum_{|\alpha|-k}\frac{k!}{\alpha!}{\bf\large
S}^\alpha{\bf\large T}^\alpha+\sum_{0\leq k\leq
m-1}(-1)^k\binom{m}{k}\sum_{|\alpha|=k+1}\frac{(k+1)!}{\alpha!}{\bf\large
S}^\alpha{\bf\large
T}^\alpha\\&&\sum_{|\alpha|=m+1}\frac{(m+1)!}{\alpha!}{\bf\large
S}^\alpha{\bf\large T}^\alpha\\&=&-\beta_m({\bf \large S,
T})+\sum_{0\leq k\leq
m-1}(-1)^k\binom{m}{k}\sum_{|\alpha|=k+1}\frac{k!(\alpha_1+...+\alpha_d)}{\alpha_1!.\alpha_2....\alpha_d!}{\bf\large
S}^\alpha{\bf\large
T}^\alpha\\&&+\sum_{|\alpha|=m+1}\frac{m!(\alpha_1+...+\alpha_d)}{\alpha_1!.\alpha_2....\alpha_d!}{\bf\large
S}^\alpha{\bf\large T}^\alpha
\\&=&-\beta_m{({\bf \large S,T})}\\&&+\sum_{1\leq j \leq
d}\sum_{0\leq k\leq
m-1}(-1)^k\binom{m}{k}\sum_{|\alpha|=k+1}\frac{k!\alpha_j}{\alpha_1!.\alpha_2!....\alpha_d!}\bigg(\\&&S_jS^{\alpha_1}...S_j^{\alpha_j-1}S_{j+1}^{\alpha_{j+1}}...S_d^{\alpha_d}
T^{\alpha_1}...T_j^{\alpha_j-1}T_{j+1}^{\alpha_{j+1}}...T_d^{\alpha_d}T_j\bigg)
\\&&+\sum_{1\leq j
\leq
d}\sum_{|\alpha|=m+1}\frac{m!\alpha_j}{\alpha_1!.\alpha_2!....\alpha_d!}S_jS^{\alpha_1}...S_j^{\alpha_j-1}S_{j+1}^{\alpha_{j+1}}...S_d^{\alpha_d}
T^{\alpha_1}...T_j^{\alpha_j-1}T_{j+1}^{\alpha_{j+1}}...T_d^{\alpha_d}T_j\\&=&
-\beta_m({\bf \large S,T})+\sum_{1\leq j \leq d}\sum_{0\leq k\leq
m-1}(-1)^k\binom{m}{k}\sum_{|\beta|=k}\frac{k!}{\beta!}S_j{\bf\large
S}^\beta{\bf T}^{\beta}T_j
\\&&+\sum_{1\leq j
\leq d}\sum_{|\alpha|=m}\frac{m!}{\beta!}S_j{\bf\large S}^\beta{\bf
T}^{\beta}T_j\\&=&-\beta_m({\bf\large S,T})+\sum_{1\leq j\leq
d}\sum_{0\leq k \leq
m}(-1)^k\binom{m}{k}\sum_{|\beta|=k}\frac{k1}{\beta!}S_j{\bf\large
S}^\beta{\bf\large T}^\beta T_j\\&=&-\beta_{m}({\bf \large S},{\bf
\large T})+\sum_{1\leq j \leq d}S_j\beta_{m}({\bf \large S},{\bf
\large T})T_j.
\end{eqnarray*}
\end{proof}
\par \vskip 0.2 cm \noindent For $k, n\in \mathbb{N}$ denote
the (descending Pochhammer) symbol by $ n^{(k)}$, i.e. $$
n^{(k)}=\left\{
\begin{array}{lll}
 0 ,\;\;\hbox{if}\;\;n=0\\
 \\
0 \;\hbox{if}\; n>0\;\;\hbox{and}\;\;k>n \\
\\
\binom{n}{k}k!\; \;\hbox{if}\;\; n>0 \;\;\hbox{ and}\;\;k\leq n.
\end{array}
  \right.$$
\begin{proposition}
Let  ${\bf \large S}=(S_1,S_2,...,S_d) \in
\mathcal{B}(\mathcal{H})^d $  and $ {\bf \large
T}=(T_1,T_2,...,T_d)\in \mathcal{B}(\mathcal{H})^d $ are an
commuting  operators . Then the following properties hold:
\begin{enumerate} \item
$$\sum_{|\alpha|=n}\frac{n!}{\alpha!}{\bf\large S}^\alpha {\bf\large T}^\alpha=\sum_{0\leq k\leq n}n^{(k)}\beta_k({\bf\large S,T}),\;\;\hbox{for all}\;\; n=0,1,....$$
\item  If ${\bf\large S}$ is an left $m$-inverse of $ {\bf \large T}$, then
$$\sum_{|\alpha|=n}\frac{n!}{\alpha!}{\bf\large S}^\alpha{\bf\large T}^\alpha=\sum_{0\leq k\leq m-1}n^{(k)}\beta_k({\bf\large S,T}),\;\;\hbox{for all}\;\; n=0,1,....$$
\end{enumerate}
\end{proposition}
\begin{proof}
\begin{enumerate}
\item  We prove the statement by  indication on $n$. For $n=0,1$ the statement is true.Suppose that the statement is true for $n$.

Form the identity $$\beta_{n+1}({\bf\large
S,T})=\displaystyle\sum_{0\leq k\leq
n+1}(-1)^{n+1-k}\binom{n+1}{k}\sum_{|\alpha|=k}\frac{k!}{\alpha!}{\bf\large
S}^\alpha{\bf\large T} ^\alpha$$ it follows that

$$
\sum_{|\alpha|=n+1}\frac{(n+1)!}{\alpha!}{\bf\large S}^\alpha{\bf
\large T}^\alpha=\beta_{n+1}({\bf \large S,T})-\sum_{0\leq k\leq
n}(-1)^{n+1-k}\binom{n+1}{k}\sum_{|\alpha|=k}\frac{k!}{\alpha!}{\bf\large
S}^\alpha{\bf \large T}^\alpha
$$
By the assumption and similar calculation as in $[5]$ we obtained
\begin{eqnarray*}
\sum_{|\alpha|=n+1}\frac{(n+1)!}{\alpha!}{\bf\large S}^\alpha{\bf
\large T}^\alpha &=&\beta_{n+1}({\bf \large S,T})-\sum_{0\leq k\leq
n}(-1)^{n+1-k}\binom{n+1}{k}\sum_{0\leq j\leq
k}k^{(j)}\beta_j({\bf\large S,T})\\&=&\sum_{0\leq k \leq
n+1}(n+1)^{(k)}\beta_k({\bf\large S,T}).
\end{eqnarray*}

\item The result follows immediately from the fact that if ${\bf \large S}$ is a  left $m$-inverse of ${\bf\large T}$
then $\beta_k({\bf\large S,T})=0$ for all $k\geq m$ (see Lemma 4.1).
\end{enumerate}
\end{proof}
\begin{remark}
If ${\bf\large T}=(T_1,T_2) \in \mathcal{B}(\mathcal{H})^2$ is an
left $2$-invertible with the left $2$-inverse
 ${\bf\large S}=(S_1,S_2) \in \mathcal{B}(\mathcal{H})^2$, then
 $$\sum_{\alpha_1+\alpha_2=n}\frac{n!}{\alpha_1!\alpha_2!}{\bf\large S}^\alpha{\bf\large T}^\alpha=n\big(S_1T_1+S_2T_2\big)-(n-1)I_{\mathcal{H}}$$

\end{remark}
 \noindent\begin{theorem} Let ${\bf\large T}=(T_1,T_2,...,T_d)\in
\mathcal{B}(\mathcal{H})^d$. If ${\bf \large T}$ possesses a left
$m$- inverse ${\bf \large S}=(S_1,S_2,...,S_d) \in
\mathcal{B}(\mathcal{H})^d$, then  the following statements hold:
\par \vskip 0.2 cm \noindent
(1)\; $[0]\not\subset \sigma_{ap}({\bf \large T}),$\par \vskip 0.2
cm \noindent (2)\; If $\lambda=(\lambda_1,...,\lambda_d) \in
\sigma_{ap}({\bf \large T})$, then
$(\displaystyle\frac{1}{d.\lambda_1},...,\displaystyle\frac{1}{d.\lambda_d})\in
\sigma_{ap}({\bf\large S})$,\vskip 0.2 cm \noindent (3)\;
 If $\lambda=(\lambda_1,...,\lambda_d) \in \sigma_{p}({\bf \large T})$, then
$(\displaystyle\frac{1}{d.\lambda_1},...,\displaystyle\frac{1}{d.\lambda_d})\in
\sigma_{p}({\bf\large S})$.

\end{theorem}
\begin{proof}
\par \vskip 0.2 cm \noindent (1)\; Suppose contrary to our claim that $[0]\subset \sigma_{ap}({\bf \large T})$ and let
$\lambda=(\lambda_1,...,\lambda_d) \in [0]$. Then there exists a
sequence $(x_n)_n \in \mathcal{H}$ such that  $$\Vert
x_n\Vert=1\;\;\hbox{ and}\;\; (T_j-\lambda_j)x_n\longrightarrow 0
\;\hbox{as}\;\;n\longrightarrow +\infty \;\hbox{for}\;j=1,2,...,d.$$
For $\alpha_j\geq 1$ we deduce that $$
(T_j^{\alpha_j}-\lambda_j^{\alpha_j})x_n\longrightarrow 0
\;\hbox{as}\;\;n\longrightarrow +\infty \;\hbox{for}\;j=1,2,...,d$$
 which mean that $({\bf\large T}^\alpha -\lambda^\alpha )x_n\longrightarrow
 0$ and hence,$\bigg({\bf\large S}^\alpha{\bf\large
T}^\alpha-\lambda^\alpha{\bf\large S}^\alpha\bigg)x_n\longrightarrow
0.$\par\vskip 0.2 cm \noindent Now, we get
\begin{eqnarray*}
\bigg({\bf\large S}^\alpha{\bf\large
T}^\alpha-\lambda^\alpha{\bf\large S}^\alpha\bigg)x_n\longrightarrow
0&\Longrightarrow& \sum_{0\leq k \leq
m}(-1)^{m-k}\binom{m}{k}\sum_{|\alpha|=k}\frac{k!}{\alpha!}\bigg({\bf\large
S}^\alpha{\bf\large T}^\alpha-\lambda^\alpha{\bf\large
S}^\alpha\bigg)x_n\longrightarrow 0\\&\Longrightarrow &
(-1)^mx_n+\sum_{1\leq k \leq
m}(-1)^{m-k}\binom{m}{k}\sum_{|\alpha|=k}\frac{k!}{\alpha!}\bigg(\prod_{1\leq
j\leq d}\lambda_j^{\alpha_j}\bigg){\bf\large S}^\alpha
x_n\longrightarrow 0\\&\Longrightarrow& x_n\longrightarrow
0\;\;\hbox{as}\;\; n\longrightarrow 0\;\;(\hbox{since}\;\; \lambda
\in [0]),
\end{eqnarray*}
  which is impossible.\par \vskip 0.2 cm \noindent (2)\; Let $\lambda=(\lambda_1,...,\lambda_d) \in
\sigma_{ap}({\bf \large T})$,  then there exists a sequence $(x_n)_n
\in \mathcal{H}$ such that
$$\Vert
x_n\Vert=1\;\;\hbox{ and}\;\; (T_j-\lambda_j)x_n\longrightarrow 0
\;\hbox{as}\;\;n\longrightarrow +\infty \;\hbox{for}\;j=1,2,...,d.$$

\begin{eqnarray*}
\bigg({\bf\large S}^\alpha{\bf\large
T}^\alpha-\lambda^\alpha{\bf\large S}^\alpha\bigg)x_n\longrightarrow
0&\Longrightarrow& \sum_{0\leq k \leq
m}(-1)^{m-k}\binom{m}{k}\sum_{|\alpha|=k}\frac{k!}{\alpha!}\bigg({\bf\large
S}^\alpha{\bf\large T}^\alpha-\lambda^\alpha{\bf\large
S}^\alpha\bigg)x_n\longrightarrow 0\\&\Longrightarrow & \sum_{0\leq
k \leq
m}(-1)^{m-k}\binom{m}{k}\sum_{|\alpha|=k}\frac{k!}{\alpha!}\bigg(\prod_{1\leq
j\leq d}\lambda_j^{\alpha_j}\bigg){\bf\large S}^\alpha
x_n\longrightarrow 0\\&\Longrightarrow&
\bigg(I_{\mathcal{H}}-\sum_{1\leq j \leq d}\lambda_jS_j
\bigg)^mx_n\longrightarrow 0\;\;\hbox{as}\;\;n\longrightarrow
\infty.
\end{eqnarray*}
On the other hand, we have
$$\bigg(I_{\mathcal{H}}-\sum_{1\leq j \leq d}\lambda_jS_j
\bigg)^mx_n\longrightarrow 0\Longrightarrow \bigg(\sum_{1\leq j\leq
d}\lambda_j\bigg(\frac{1}{d.\lambda_j}-S_j\bigg)
\bigg)^mx_n\longrightarrow 0.$$ From Proposition 3.1, it follows
that there exists a sequence $(x_n)_n \subset \mathcal{H}$ such that
$\|x_n\|=1$ and $$\lim_{n\longrightarrow \infty}\|\big(
\frac{1}{d}.\frac{1}{\lambda_j}-S_j\big)x_n\|=0.
$$
From this $\bigg(\displaystyle\frac{1}{d.\lambda_j}\bigg)_{1\leq j
\leq d}\in \sigma_{ap}({\bf \large S}).$ \par\vskip 0.2 cm \noindent
(3)\; The argument is similar to one given in (2).
\end{proof}
\noindent The proof of the following theorem is similar to the proof
of Theorem 4.1, so we omit it.
\begin{theorem}
Let ${\bf\large T}=(T_1,T_2,...,T_d)\in \mathcal{B}(\mathcal{H})^d$.
If ${\bf \large T}$ possesses a right $m$- inverse ${\bf \large
R}=(R_1,R_2,...,R_d) \in \mathcal{B}(\mathcal{H})^d$, then  the
following statements hold:\par\vskip 0.2 cm \noindent (1)\;
 $[0]\not\subset \sigma_{ap}({\bf\large R}).$
 \par\vskip 0.2 cm \noindent (2)\; If $\lambda=(\lambda_1,...,\lambda_d)\in \sigma_{ap}({\bf\large R})$, then $ \big(\displaystyle\frac{1}{d.\lambda_1},...,\frac{1}{d.\lambda_d}\big)\in \sigma_{ap}({\bf\large
 T})$.\par\vskip 0.2 cm \noindent (3)\;If $\lambda=(\lambda_1,...,\lambda_d)\in \sigma_{p}({\bf\large R})$ ,then $ \big(\displaystyle\frac{1}{d.\lambda_1},...,\frac{1}{d.\lambda_d}\big)\in \sigma_{p}({\bf\large
 T}).$
\end{theorem}

\end{document}